\documentclass{article}
\usepackage{graphicx}
\usepackage{color}
\usepackage{amscd}
\usepackage{framed}
\usepackage{bbm}
\usepackage{amsmath,amsthm,amssymb}
\usepackage{fancyhdr,a4wide}

\usepackage{comment}
\newcommand{\mc}{\mathcal}
\newcommand{\mbb}{\mathbb}

\newcommand{\mr}{\mathrm}
\newcommand{\argmin}{\mathop{\rm argmin}\limits}

\newcommand{\vecg}{\mathbf{g}}

\newcommand{\vecu}{\mathbf{u}}
\newcommand{\vecv}{\mathbf{v}}
\newcommand{\vecw}{\mathbf{w}}
\newcommand{\vecx}{\mathbf{x}}

\newcommand{\vecz}{\mathbf{z}}
\newcommand{\vecX}{\mathbf{X}}

\newcommand{\vecbeta}{\boldsymbol \beta}

\newcommand{\vecdelta}{\boldsymbol \delta}
\newcommand{\vecvarrho}{\boldsymbol \varrho}
\newcommand{\vectau}{\boldsymbol \tau}

\newcommand{\vecmu}{\boldsymbol \mu}
\usepackage{bm}

\usepackage{algorithm}
\usepackage{algorithmic}
\renewcommand{\algorithmicrequire}{\textbf{Input:}}
\renewcommand{\algorithmicensure}{\textbf{Output:}}

\numberwithin{equation}{section}
\newtheorem{theorem}{Theorem}[section]
\newtheorem{proposition}{Proposition}[section]
\newtheorem{remark}{Remark}[section]
\newtheorem{lemma}{Lemma}[section]

\newtheorem{definition}{Definition}[section]
\newtheorem{assumption}{Assumption}[section]

\begin{document}
\title{Adversarial Robust Weighted Huber Regression}
\author{Takeyuki Sasai
\thanks{Department of Statistical Science, The Graduate University for Advanced Studies, SOKENDAI, Tokyo, Japan. Email: sasai@ism.ac.jp}
\and Hironori Fujisawa
\thanks{The Institute of Statistical Mathematics, Tokyo, Japan.
Department of Statistical Science, The Graduate University for Advanced Studies, SOKENDAI, Tokyo, Japan.
Center for Advanced Integrated Intelligence Research, RIKEN, Tokyo, Japan. Email:fujisawa@ism.ac.jp}
}
\maketitle
\begin{abstract}
	We consider a robust estimation of  linear regression coefficients. In this note, we focus on the case where the covariates are sampled from an $L$-subGaussian distribution with unknown covariance, the noises are sampled from a distribution with a bounded absolute moment and both covariates and noises may be contaminated by an adversary.
	We derive an estimation error bound, which depends on the stable rank and the condition number of the covariance matrix of covariates with a polynomial computational complexity of estimation.
\end{abstract}

\section{Introduction}
	In the present paper,  we study linear regression problem when outputs and inputs are contaminated by adversary:
	\begin{align}
	\label{model:adv}
	y_i = (\vecx_i+\vecvarrho_i)^\top\vecbeta^*+\xi_i+\sqrt{n}\theta_i,\quad  i=1,\cdots,n,
	\end{align}
	where $\left\{\vecx_i\right\}_{i=1}^n$ is a sequence of covariate vector, $\vecbeta^* \in \mbb{R}^d$ is the true coefficient vector and  $\left\{\xi_i\right\}_{i=1}^n$ is a sequence of noise. 	We admit adversary to inject arbitrary values into arbitral $o$ samples of $(y_i,\vecx_i)_{i=1}^n$. Let $\mc{O}$ be the index  set of the injected samples and $\vecvarrho_i= (0,\cdots,0)^\top$ and $\theta_i=0$ for $i \in \mc{I}= (1,\cdots,n)\setminus \mc{O}$. 
	The difficulty is that $\{\vecvarrho_i\}_{i=1}^n$ and $\{\theta_i\}_{i=1}^n$ can be arbitral values and they are allowed to be correlated freely among them and correlated with $\{\vecx_i\}_{i=1}^n$ and $\{\xi_i\}_{i=1}^n$.
	We note that $\{\vecx_i,\xi_i\}_{i\in\mc{I}}$ is no longer i.i.d. sequence because we allow adversary freely to select samples for contamination.

	\cite{KLiKotMek2018Efficient,DiaKonSte2019Efficient,BakPra2021Robust, PenJogLoh2020robust,CheAraTriJorFlaBari2020Optimal} considered the estimation problems of $\vecbeta^*$ from \eqref{model:adv}.
	When the covariance matrix of covariates is unknown and the covariates are drawn from a heavy-tailed distribution,
	\cite{BakPra2021Robust, PenJogLoh2020robust,CheAraTriJorFlaBari2020Optimal}
	derived sharp error bounds.
	We consider the case where the covariance matrix of covariates is unknown and the covariates are drawn from an $L$-subGaussian distribution.
	Our finding is that even when the covariance is unknown, the estimation error bound can become sharper when covariates are drawn from an $L$-subGaussian distribution.
	Our method is basically the same as Algorithm 1 of \cite{PenJogLoh2020robust} and the main difference is preprocessing for the covariate.

	Define the maximum and minimum eigenvalue of the matrix $M$ as $\lambda_{\max}(M)$ and $\lambda_{\min}(M)$, respectively and define the condition number of $\Sigma^\frac{1}{2}=\kappa(\Sigma^\frac{1}{2}) = \lambda_{\max}(\Sigma^\frac{1}{2})/\lambda_{\min}(\Sigma^\frac{1}{2})$.
	Our  result is  the following. 	For the precise statements, see Section \ref{sec:results}.
	\begin{theorem}
		\label{theoreminformal0}
		Suppose that $\left\{\vecx_i \right\}_{i=1}^n$ is a sequence with i.i.d. random vectors drawn from an $L$-subGaussian distribution with mean zero and   with covariance matrix $\Sigma$, where $\Sigma$ is unknown. 
		Suppose that $\left\{\xi_i\right\}_{i=1}^n$ is a sequence with i.i.d. random variables drawn from a distribution whose absolute moment is bounded by $\sigma$. 
		 Assume that $\left\{\vecx_i \right\}_{i=1}^n$ and $\left\{\xi_i\right\}_{i=1}^n$ are independent.
		Then, with polynomial computational time, we can construct $\hat{\vecbeta}$ such that 
		\begin{align}
		\label{t:informal1}
		\|\hat{\vecbeta}-\vecbeta^*\|_2 \leq   c(L,\sigma) \times\frac{\kappa(\Sigma^\frac{1}{2})}{\lambda_{\min}(\Sigma^\frac{1}{2})}\times\left(\sqrt{\frac{r_\Sigma}{n}}+\sqrt{\frac{\log(1/\delta)}{n}}+\sqrt{\frac{o}{n}}\right),
		\end{align}
		where  $c(L,\sigma)$ is some constant depending on $L$ and $\sigma$ and $r_\Sigma = \mr{Trace}(\Sigma)/\lambda_{\max}(\Sigma)$,
		with high probability probability at least $1-5\delta$. 
	\end{theorem}
	\begin{remark}
		We assume the independence of $\left\{\vecx_i \right\}_{i=1}^n$ and $\left\{\xi_i\right\}_{i=1}^n$ in the theorem above.
		Strictly speaking, our theorem requires slightly weaker condition than the independence. See Theorem \ref{t:1:main}.
	\end{remark}
	We note that $r_\Sigma$ is sometimes called stable rank \cite{Ver2018High} and we easily see $r_\Sigma\leq d$.
	The term including $d$ in error bounds of \cite{BakPra2021Robust,CheAraTriJorFlaBari2020Optimal,PenJogLoh2020robust} is $\tilde{O}(d)$, but in this paper, we can remove the logarithmic dependence of $d$ when covariates are drawn from an $L$-subGaussian distribution.

	In Section \ref{se:om}, we provide our estimation method. 
	In Section \ref{sec:results}, we state our main results.
	In Section \ref{sec:key}, we state key propositions without proof.
	In Section \ref{sec:pmt}, we give the proofs of main results using the propositions stated in Section \ref{sec:key}.
	In Section \ref{sec:maindet}, we give the proofs of propositions in  Section \ref{sec:key}.
\section{Method}
	\label{se:om}
	Define $\varepsilon = \frac{o}{n}$ and $\vecX_i = \vecx_i+\vecvarrho_i$.
To estimate $\vecbeta^*$ in \eqref{model:adv}, 
we propose the following algorithm (Algorithm \ref{outmethod}).
\begin{algorithm}
	\caption{TWO STEP WEIGHTED HUBER REGRESSION}
	\begin{algorithmic}[1]
			\label{outmethod}
		\renewcommand{\algorithmicrequire}{\textbf{Input:}}
		\renewcommand{\algorithmicensure}{\textbf{Output:}}
		\REQUIRE $\left\{y_i,\vecX_i\right\}_{i=1}^n$, $\varepsilon$ and the tuning parameter $\lambda_o$
		\ENSURE  $\hat{\vecbeta}$
		\STATE $\left\{ \hat{w}_i\right\}_{i=1}^n \leftarrow \text{COMPUTE-WEIGHT}(\left\{\vecX_i\right\}_{i=1}^n,\varepsilon)$
		\STATE $\left\{ \hat{w}_i'\right\}_{i=1}^n \leftarrow \text{THRESHOLDING}(\left\{ \hat{w}_i\right\}_{i=1}^n )$
		\STATE $\hat{\vecbeta} \leftarrow \text{WEIGHTED-HUBER-REGRESSION}(\left\{y_i,\vecX_i,\hat{w}'_i\right\}_{i=1}^n, \lambda_o)$
	\end{algorithmic} 
\end{algorithm}

In the following Sections \ref{sec:cw}, \ref{sec:t} and \ref{sec:whr}, we describe the roles of the procedures COMPUTE-WEIGHT, THRESHOLDING and 
WEIGHTED-HUBER-REGRESSION, respectively.

\subsection{COMPUTE-WEIGHT}
\label{sec:cw}
COMPUTE-WEIGHT reduces the effects of outliers for covariates.
We require COMPUTE-WEIGHT to compute the weight vector $\hat{w}= (\hat{w}_1,\cdots, \hat{w}_n)$ such that the following quantity is sufficiently small:
\begin{align}
	\label{ine:adv-spect}
	\left|\sum_{i \in \mc{O}}\hat{w}_iu_i \vecX_i^\top(\vecbeta^*-\hat{\vecbeta}) \right|
\end{align}
for $\vecu$ such that $\|\vecu\|_\infty \leq c$ for sum numerical constant $c$. In the proof of our main proposition (Proposition \ref{t:main}), the evaluation of \eqref{ine:adv-spect} appears.
In the present paper, we use COMPUTE-WEIGHT as a variant of Algorithm 1 of \cite{DalMin2020All} and the algorithm satisfies our aim.
Algorithm 1 of \cite{DalMin2020All} computes a weight sequence $\{w_i\}_{i=1}^n$ such that $\sum_{i=1}^n w_i\vecX_i$
is close to the true mean in $\ell_2$ norm.
As we refer in the Introduction, there are many algorithms to estimate mean with existence on outliers and many of them have the properties such that \eqref{ine:adv-spect}. However, the algorithm in \cite{DalMin2020All} derived dimension free bound at the first time with polynomial time conplexity and our slightly modified algorithm inherit the properties. We note that, the algorithm in \cite{DalMin2020All} have another fascinating points such that high breakdown point and so on.

Here, we introduce Algorithm 1 of \cite{DalMin2020All}.
Define the contamination ratio $\varepsilon  = o/n$ and 
define the probability simplex $\Delta^{n-1}$ as
\begin{align}
	\Delta^{n-1} = \left\{\vecw \in [0,1]^n: \sum_{i=1}^nw_i =1, \quad \|\vecw\|_\infty\leq \frac{1}{n(1-\varepsilon)}\right\}.
\end{align}
The algorithm introduced in \cite{DalMin2020All} is the following.
For some matrix $M$, define $\lambda_{\max}(M)$ is the maximum singular value of $M$.
\begin{algorithm}[H]
	\caption{Algorithm 2 of \cite{DalMin2020All}}
	\label{alg:cw0}
	\begin{algorithmic}
	\REQUIRE{data $\vecX,\cdots, \vecX_n \in \mbb{R}^d$ and $\varepsilon$.}
	\ENSURE{parameter estimate $\hat{\vecmu}$} and weight estimate $\hat{\vecw} = \{\hat{w}_1,\cdots,\hat{w}_n\}$\\
	{\bf Initialize}: compute $\hat{\vecmu}_0$ as the minimizer of $\sum_{i=1}^n \|\vecX_i-\vecmu\|_2$\\

	Set $K=\max\left\{0,\frac{\log(4p)-2\log (\varepsilon(1-2\varepsilon))}{2\log(1-2\varepsilon)-\log \varepsilon -\log(1-\varepsilon)}\right\}$.

	{\bf For} {$k=1:K$}\\
		\ \ \ \ \ Compute current weights:\\
		\ \ \ \ \ $\hat{\vecw}_{k}=(w_{1,k},\cdots,w_{n,k})^\top \in \argmin_{\vecw \in \Delta^{n-1}} \lambda_{\max}\Big(\sum_{i=1}^n w_i (\vecX_i-\hat{\vecmu}_{k-1})(\vecX_i-\hat{\vecmu}_{k-1})^\top \Big)$.\\
		\ \ \ \ \ Update the estimator: $\hat{\vecmu}_{k} = \sum_{i=1}^n\hat{w}_{i,k}\vecX_i$.\\
	 {\bf return}	$\hat{\vecmu}_K$ and $\hat{\vecw}^K$.    
\end{algorithmic}
\end{algorithm}
We define COMPUTE-WEIGHT as following.
\begin{algorithm}[H]
	\caption{COMPUTE-WEIGHT}
	\label{alg:cw}
	\begin{algorithmic}
	\REQUIRE{data (covariates) $\vecX_1,\cdots, \vecX_n \in \mbb{R}^d$ and $\varepsilon$.}
	\ENSURE{ weight estimate $\hat{\vecw} = \{\hat{w}_1,\cdots,\hat{w}_n\}$}\\
	Compute weights:\\
	$\hat{\vecw}=(w_1,\cdots,w_{n})^\top \in \argmin_{\vecw \in \Delta^{n-1}} \lambda_{\max}\Big(\sum_{i=1}^n w_i \vecX_i\vecX_i^\top \Big)$.\\
	 {\bf return}	$\hat{\vecw}$.    
\end{algorithmic}
\end{algorithm}
\medskip 
The only difference is that, Algorithm 2 of \cite{DalMin2020All}
estimate the unknown true mean from data.
However, COMPUTE-WEIGHT does not need to estimate the mean because, in our setting, the mean of the covariates $\{\vecx_i\}_{i=1}^n$ is known and
the loop in Algorithm \ref{alg:cw0} is not needed.
\subsection{THRESHOLDING}
\label{sec:t}
THRESHOLDING is a simple discretization of $\left\{ \hat{w}_i\right\}_{i=1}^n$ and makes it easy to analyze the estimator.
\begin{algorithm}[H]
	\caption{THRESHOLDING}
	\label{alg:t}
	\begin{algorithmic}
	\REQUIRE{weight vector $\hat{\vecw} = \{\hat{w}_i\}_{i=1}^n$.}
	\ENSURE{discretized weight vector $\hat{\vecw}'= \{\hat{w}'_i\}_{i=1}^n$}\\
	{\bf For} {$i=1:n$}\\
		\ \ \ \ \ {\bf if} $\hat{w}_i \geq \frac{1}{2n}$\\
		\ \ \ \ \ \ \ \ \ \ $\hat{w}'_i = \frac{1}{n}$\\
		\ \ \ \ \ {\bf else}\\
		\ \ \ \ \ \ \ \ \ \ $\hat{w}'_i = 0$\\
	 {\bf return}	$\hat{\vecw}'$.    
\end{algorithmic}
\end{algorithm}
From the following lemma, we see that 
the number of $\hat{w}_i,\,i=1,\cdots,n$ rounded at zero is at most  $2o$. 
Let $I_{<}$ and $I_{\geq}$ be  the sets of the indices such that $w_i <  \frac{1}{2n}$  and $w_i \geq \frac{1}{2n}$, respectively. 
\begin{lemma}
	\label{l:w2}
	Suppose $n>2o$. For any $w_i \in \Delta_{n,\varepsilon}$, we have $|I_{<}| \leq 2o$.
\end{lemma}
 \begin{proof}
 	We assume $|I_{<}| > 2o$ and $o>1$, and then we derive a contradiction.
 	From the constraint about $w_i$,  we have $0\leq w_i \leq  \frac{1}{\left(1-\varepsilon\right)n}$ and 
 	\begin{align*}
 	\sum_{i=1}^n w_i = \sum_{i \in I_{<}} w_i+ \sum_{i \in I_{\geq}} w_i &\leq |I_{_<}| \times \frac{1}{2n} + (n-|I_{<}|) \times \frac{1}{\left(1-\varepsilon\right)n}\\
 	& = 2o \times \frac{1}{2n} + (|I_{<}|-2o) \times \frac{1}{2n} + (n-2o) \times \frac{1}{\left(1-\varepsilon\right)n} +(2o -|I_{<}|) \times \frac{1}{\left(1-\varepsilon\right)n}\\
 	&=2o \times \frac{1}{2n} + (n-2o) \times \frac{1}{\left(1-\varepsilon\right)n} +(|I_{<}|-2o) \times\left(\frac{1}{2n}-\frac{1}{\left(1-\varepsilon\right)n}\right)\\
 	&<2o \times \frac{1}{2n} + (n-2o) \times \frac{1}{\left(1-\varepsilon\right)n}\\
 	& =\frac{1-\varepsilon-\varepsilon^2}{1-\varepsilon}<1.
  	\end{align*}
 	This is a contradiction to the constraint about $\{w_i\}_{i=1}^n$, $\sum_{i=1}^nw_i=1$.
\end{proof}
In the proof of Proposition \ref{p:main:out}, we prove that 
\begin{align}
	\label{ine:adv-spect2}
	\left|\sum_{i \in \mc{O}}\hat{w}_iu_i \vecX_i^\top\vecv \right| \text{ and } 	\left|\sum_{i \in \mc{O}}\hat{w}'_iu_i \vecX_i^\top\vecv \right|
\end{align}
are not different in the view of the convergence rate.

\subsection{WEIGHTED-HUBER-REGRESSION}
\label{sec:whr}
WEIGHTED-HUBER-REGRESSION is a simple regression using Huber loss.
We consider the following optimization problem.
\begin{align}
\label{opt:Hp}
\hat{\vecbeta}= \argmin_{\vecbeta  \in \mbb{R}^d} \sum_{i=1}^n \lambda_o^2 H\left(\hat{w}'_in\frac{y_i-\vecX_i^\top\vecbeta}{\lambda_o\sqrt{n}}\right),
\end{align}
where $H(t)$ is the Huber loss function
\begin{align}
H(t) = \begin{cases}
|t| -1/2 & (|t| > 1) \\
t^2/2  & (|t| \leq 1)
\end{cases}.
\end{align}
and let 
\begin{align}
	h(t) =	\frac{d}{dt} H(t) =   \begin{cases}
	t\quad &(|t| \leq 1)\\
	\mr{sgn}(t)\quad &(|t|  >1)
\end{cases}.
\end{align}

Many studies (\cite{NguTra2012Robust}, \cite{DalTho2019Outlier} and \cite{CheZho2020Robust} and so on ) imply that Huber loss is effective for linear regression under the existence of heavy-tailed noise or adversarial noise. 
As in previous studies, the present paper use the properties.
\section{Result}
\label{sec:results}
We introduce an $L$-subGaussian random vector, which appeared in \cite{Men2016Upper,MenZhi2020Robust,Tho2020Outlier} and other studies.
\begin{definition}[$L$-subGaussian random vector]
	A random vector $\vecx \in \mbb{R}^d$ with the mean $\mbb{E}\vecx = \vecmu$ is said to be an $L$-subGaussian if for every $\vecv \in \mbb{R}^d$ and every $p\geq 2$,
	\begin{align}
		\|\langle \vecx-\vecmu,\vecv\rangle\|_{\psi_2}\leq L\left(\mbb{E}|\langle \vecx-\vecmu,\vecv\rangle|^2\right)^\frac{1}{2},
	\end{align}
	where the norm $\|\cdot\|_{\psi_2}$ is defined in Definition \ref{d:orlicz}.
\end{definition}
\begin{definition}[$\psi_2$-norm]
	\label{d:orlicz}
	Let $f$ be defined on  a probability space. Set
	\begin{align}
		\|f\|_{\psi_2}:=	\inf\left[ \eta>0\,:\, \mbb{E}\exp(f/\eta)^2\leq 2\right] < \infty.
	\end{align}	
\end{definition}
To derive our result, suppose the following assumption.
\begin{assumption}
	\label{a:1}
	Assume that 
	\begin{itemize}
		\item[(i)] $\vecx_i = \Sigma^\frac{1}{2} \vecz_i$, where $\{\vecz_i\}_{i=1}^n$  is a sequence with independent random vectors drawn from an $L$-subGaussian distributions with mean zero and $\Sigma:=\mbb{E}\vecx_i \vecx_i^\top$.
  \item[(ii)] $\lambda_{\min}(\Sigma^\frac{1}{2})> 0$.
		\item[(iii)] $\{\xi_i\}_{i=1}^n$ is a sequence with independent random variables from a distribution whose absolute moment is bounded by $\sigma$.
		\item[(iv)] for $i=1,\cdots,n$, $\mbb{E}h\left(\frac{\xi_i}{\lambda_o\sqrt{n}}\right) \times \vecx_i=0$.
	\end{itemize}
\end{assumption}
\begin{remark}
	The condition (iii) in Assumption \ref{a:1} is a weaker condition than the independence between $\{\xi_i\}_{i=1}^n$ and $\{\vecx_i\}_{i=1}^n$ .
\end{remark}

Under Assumption \ref{a:1}, we have the following theorem.
\begin{theorem}
	\label{t:1:main}
	Suppose that Assumption \ref{a:1} holds. Consider the optimization problem \eqref{opt:Hp}. 
	Let 
	\begin{align}
		r_1= c_1(L)  \times\lambda_o\sqrt{n}\times\frac{\kappa(\Sigma^\frac{1}{2})}{\lambda_{\min}(\Sigma^\frac{1}{2})}\times\left(\sqrt{\frac{r_\Sigma}{n}}+\sqrt{\frac{\log(1/\delta)}{n}}+\sqrt{\frac{o}{n}}\right)
		\end{align}
		with $c_1(L)$ is some constant depending on $L$.
	Suppose that  $\frac{r_\Sigma+\log(1/\delta)}{n}\leq 1$, $n>2o$, $\lambda_o\sqrt{n} \geq 72 L^4 \sigma$ and $r_1 \leq \frac{1}{2\sqrt{3} L^2\sqrt{\|\Sigma\|_{\mr{op}}}}$.
	Then, the optimal solution $\hat{\vecbeta}$ satisfies $\|\hat{\vecbeta} -\vecbeta^*\|_2 < r_1$
	with probability at least $1-5\delta$.
\end{theorem}

\section{Key propositions}
\label{sec:key}
First, we introduce our main proposition.
The main proposition is stated in a `deterministic' style.
Propositions \ref{p:main1}, \ref{p:main:out}, \ref{p:main:out2} and  \ref{p:main:sc} insist that the conditions in the main proposition are satisfied with high probability.

For $v \in \mbb{R}^d$, let
\begin{align}
	r^\circ_{i,\vecv} =\frac{\vecx_i^\top\vecbeta^* +\xi_i-\vecx_i^\top \vecv}{\lambda_o \sqrt{n}}.
\end{align}
\begin{proposition}
\label{t:main}
	Consider the optimization problem \eqref{opt:Hp}.
	For any vector $\vecv\in r\mbb{S}^{d-1}$  and $\vecu \in \mbb{R}^n$ such that $\|\vecu\|_\infty \leq c$, where $c$ is some numerical constant and $r$ is the number defined in \eqref{ine:detassumption6}, suppose that
	\begin{align}
		\label{ine:detassumption1}
		\left|\lambda_o\sqrt{n}\sum_{i=1}^n\frac{1}{n}h\left(\frac{\xi_i}{\lambda_o\sqrt{n}}\right)\vecx_i^\top \vecv\right|&\leq c_1\|\vecv\|_2,\\
		\label{ine:detassumption2}
		\left|\lambda_o\sqrt{n}\sum_{i \in \mc{O}\cup (I_<\cap\mc{I})}\frac{1}{n}u_i\vecx_i^\top \vecv\right|&\leq c_2\|\vecv\|_2,\\
		\label{ine:detassumption4}
		\left|\lambda_o\sqrt{n}\sum_{i \in \mc{O}}\hat{w}'_i u_i \vecX_i^\top \vecv\right|&\leq c_3\|\vecv\|_2,\\
		\label{ine:detassumption5}
		\lambda_o\sqrt{n}\sum_{i =1}^n\frac{1}{n}\left\{-h\left(r^\circ_{i,\vecbeta^*+\vecv}\right) +h\left(r^\circ_{i,\vecbeta^*}\right) \right\}
		\vecx_i^\top\vecv&\geq c_4\|\vecv\|_2^2-c_5\|\vecv\|_2 -c_6,
	\end{align}
	where $c_1,\,c_2,\, c_3, c_4,\, c_5,\,c_6$ are some positive numbers.
	Let 
	\begin{align}
	\label{ine:detassumption6}
	2 \times \frac{c_1+c_2+c_3+c_5+\sqrt{c_4c_6}}{c_4} =  r
	\end{align}
	for some positive number $r$.
	Then, we have
	\begin{align}
		\|\vecbeta^*-\hat{\vecbeta}\|_2 \leq r.
	\end{align}
\end{proposition}

Let $c(L)$ be a numerical constant depending on $L$.
\begin{proposition}
	\label{p:main1}
	Suppose that Assumption \ref{a:1} holds. For any $\vecv \in \mbb{S}^{d-1}$, we have
	\begin{align}
		\left|\frac{1}{n}\sum_{i =1}^n h\left(\frac{\xi_i }{\lambda_o\sqrt{n}}\right)\vecx_i^\top \vecv\right| \leq c(L)\|\Sigma\|_{\mr{op}} \left(\sqrt{\frac{r_\Sigma}{n}} +\sqrt{\frac{\log(1/\delta)}{n}} \right)\|\vecv\|_2.
	\end{align}
	with probability at least $1-\delta$.
\end{proposition}

\begin{proposition}
	\label{p:main:out}
	For any $m$-sparse vector $\vecu= (u_1,\cdots,u_n)$ such that $\|\vecu\|_\infty \leq c$, where $c$ is a numerical constant and for any $\vecv \in \mbb{R}^d$, we have
	\begin{align}
		\left|\sum_{i \in \mc{O}}\hat{w}'_iu_i \vecX^\top\vecv \right| \leq c(L)\sqrt{\|\Sigma\|_{\mr{op}}}\sqrt{\frac{o}{n}}\|\vecv\|_2.
	\end{align}
	with probability at least $1-\delta$. 
\end{proposition}
Define $I_m$ as an index set such that $|I_m| = m$, where $|I|$ is the number of the elements of an index set $I$.
\begin{proposition}
	\label{p:main:out2}
	Suppose that Assumption \ref{a:1} holds. 
	For any $m$-sparse vector $\vecu= (u_1,\cdots,u_n)$ such that $\|\vecu\|_\infty \leq c$, where $c$ is a numerical constant and for any $\vecv \in \mbb{R}^d$, we have
	\begin{align}
		\left|\sum_{i \in I_m}\frac{1}{n}u_i \vecx_i^\top\vecv \right| \leq c(L)\sqrt{\|\Sigma\|_{\mr{op}}} \sqrt{\frac{m}{n}}\|\vecv\|_2,
	\end{align}
	with probability at least $1-\delta$.
\end{proposition}

\begin{proposition}
	\label{p:main:sc}
	Suppose that Assumption \ref{a:1} holds. 
	Let  
	\begin{align}
	\mc{R} = \left\{ \vecv \in \mbb{R}^d\, |\,  \|\vecv\|_2 = r \right\},
	\end{align}
	where $r$ is a number such that $0\leq r \leq \frac{1}{4\sqrt{3} L^4\sqrt{\|\Sigma\|_{\mr{op}}}}$.
	Assume that $\lambda_o  \sqrt{n} \geq 72L^4\sigma$.
	Then, with probability at least $1-\delta$, for any $\vecv \in \mc{R}$, we have
	\begin{align}
		\label{ine:sc}
		&\sum_{i =1}^n\frac{\lambda_o}{\sqrt{n}}\left\{-h\left(r_{i,\vecbeta^*+\vecv}\right) +h\left(r_{i,\vecbeta^*}\right) \right\}
		\vecx_i^\top\vecv\nonumber\\
		&\geq  \frac{\lambda^2_{\min}(\Sigma^\frac{1}{2})}{3}\|\vecv\|_2^2-c(L)\left(\sqrt{\|\Sigma\|_{\mr{op}}} \sqrt{\frac{r_\Sigma}{n}}+\sqrt{\frac{\log(1/\delta)}{n}} \right)\|\vecv\|_2 -c(L)\frac{\log(1/\delta)}{n}.
	\end{align}
\end{proposition}

\section{Proofs of the main theorems}
\label{sec:pmt}
\begin{proof}[proof of Theorem \ref{t:1:main}]
	We show that \eqref{ine:detassumption1}-\eqref{ine:detassumption6}  in Proposition 4.1 are satisfied with probability at least $1-5\delta$ under Assumption 3.1. Then, we see that Theorem 3.1 holds. 

	First, we see that \eqref{ine:detassumption1}-\eqref{ine:detassumption5}
	in Proposition \ref{t:main} are satisfied by Propositions \ref{p:main1}-\ref{p:main:sc} with 
	\begin{align}
		\label{ine:res1}
		c_1 &=\lambda_o\sqrt{n}\times c(L)\|\Sigma\|_{\mr{op}} \left(\sqrt{\frac{r_\Sigma}{n}} +\sqrt{\frac{\log(1/\delta)}{n}} \right),\,c_2=c_3=c(L)\sqrt{\|\Sigma\|_{\mr{op}}} \sqrt{\frac{o}{n}},\nonumber \\
		c_4 &=\frac{\lambda^2_{\min}(\Sigma^\frac{1}{2})}{3},\,c_5 =c(L)\left(\sqrt{\|\Sigma\|_{\mr{op}}} \sqrt{\frac{r_\Sigma}{n}}+\sqrt{\frac{\log(1/\delta)}{n}} \right),\,c_6 =c(L)\frac{\log(1/\delta)}{n},
	\end{align}
	where we use Lemma \ref{l:w2}, with probability at least $1-5\delta$, taking the union bound.

	Next, we confirm \eqref{ine:detassumption6}. From \eqref{ine:res1}, we have
	\begin{align}
		2\frac{c_1+c_2+c_3+c_5+\sqrt{c_4c_6}}{c_4} \leq c_1(L)  \times\lambda_o\sqrt{n}\times\frac{\kappa(\Sigma^\frac{1}{2})}{\lambda_{\min}(\Sigma^\frac{1}{2})}\times\left(\sqrt{\frac{r_\Sigma}{n}}+\sqrt{\frac{\log(1/\delta)}{n}}+\sqrt{\frac{o}{n}}\right)=r_1
		\end{align}
		for a sufficiently large constant  $c(L)$ depending on $L$ and the proof is complete.
\end{proof}

\section{Proofs of Proposition \ref{t:main}-\ref{p:main:sc}}
\label{sec:maindet}
\subsection{Proof of Proposition \ref{t:main}}
Let 
\begin{align}
	r^{adv}_{i,\vecv} =n\hat{w}'_i\frac{y_i-\vecX_i^\top \vecv}{\lambda_o \sqrt{n}}
\end{align}
for $v \in \mbb{R}^d$.

\begin{proof}[Proof of Proposition~\ref{t:main}]
Let $\vecdelta = \hat{\vecbeta}-\vecbeta^*$ and $\vecdelta_{\eta}  = \eta\vecdelta$ for some $\eta \in [0,1]$.
For  any fixed $r>0$, we define
\begin{align}
	\mbb{B}(r) :=\left\{ \vecbeta\, :\, \| \vecbeta^*-\vecbeta\|_2 \leq r\right\}.
\end{align}
We prove $\hat{\vecbeta} \in \mbb{B}(r)$ by assuming $\hat{\vecbeta} \notin \mbb{B}(r)$ and deriving a contradiction. For $\hat{\vecbeta} \notin \mbb{B}(r)$, we can find some $\eta  \in (0,1)$ such that
\begin{align}
\label{ine:det:cont}
	\|\vecdelta_\eta\|_2 = r.
\end{align}
Let 
\begin{align}
	Q'(\eta) =\lambda_o\sqrt{n}\sum_{i=1}^n \hat{w}'_i\left\{-h\left(r^{adv}_{i,\vecbeta^*+\vecdelta_\eta}\right) +h\left(r^{adv}_{i,\vecbeta^*}\right) \right\}\vecX_i^\top \vecdelta.
\end{align}
From the proof of Lemma F.2. of \cite{FanLiuSunZha2018Lamm}, we have $\eta Q'(\eta) \leq \eta Q'(1)$ and this means
\begin{align}
\label{ine:det:1}
\eta \lambda_o\sqrt{n}\sum_{i=1}^n \hat{w}'_i\left\{-h\left(r^{adv}_{i,\vecbeta^*+\vecdelta_\eta}\right) +h\left(r^{adv}_{i,\vecbeta^*}\right) \right\}\vecX_i^\top \vecdelta 
&\leq \eta\lambda_o\sqrt{n}\sum_{i=1}^n \hat{w}'_i\left\{-h\left(r^{adv}_{i,\hat{\vecbeta}}\right) +h\left(r^{adv}_{i,\vecbeta^*}\right) \right\}\vecX_i^\top \vecdelta\nonumber\\
&\stackrel{(a)}{=}\lambda_o\sqrt{n}\sum_{i=1}^n\hat{w}'_i h\left(r^{adv}_{i,\vecbeta^*}\right)\vecX_i^\top \vecdelta_\eta
\end{align}
where (a) follows from the fact that  $\hat{\vecbeta}$ is a optimal solution of~\eqref{opt:Hp} and from the definition of $\vecdelta_\eta$.
For the right hand side of \eqref{ine:det:1}, from triangular inequality, we have
\begin{align}
	\label{ine:det:1:right}
	&\sum_{i=1}^n\hat{w}'_i h\left(r^{adv}_{i,\vecbeta^*}\right)\vecX_i^\top \vecdelta_\eta=\sum_{i=1}^n\frac{1}{n} h\left(r^{\circ}_{i,\vecbeta^*}\right)\vecx_i^\top \vecdelta_\eta+\sum_{i \in \mc{O}}\hat{w}'_i h\left(r^{adv}_{i,\vecbeta^*}\right)\vecX_i^\top\vecdelta_\eta-\sum_{i\in \mc{O}\cup(I_<\cap \mc{I})}\frac{1}{n} h\left(r^\circ_{i,\vecbeta^*}\right)\vecx_i^\top \vecdelta_\eta.
\end{align}
The left-hand side of~\eqref{ine:det:1} can be decomposed as
\begin{align}
\label{ine:det:1:left}
&\lambda_o\sqrt{n}\sum_{i=1}^n\hat{w}'_i \left\{-h\left(r^{adv}_{i,\vecbeta^*+\vecdelta_\eta}\right) +h\left(r^{adv}_{i,\vecbeta^*}\right) \right\}\vecX_i^\top \vecdelta_\eta\nonumber\\
&=\lambda_o\sqrt{n}\sum_{i=1}^n \frac{1}{n}\left\{-h\left(r^\circ_{i,\vecbeta^*+\vecdelta_\eta}\right) +h\left(r^\circ_{i,\vecbeta^*}\right) \right\}\vecx_i^\top \vecdelta_\eta + \lambda_o\sqrt{n}\sum_{i\in \mc{O}}\hat{w}'_i \left\{-h\left(r^{adv}_{i,\vecbeta^*+\vecdelta_\eta}\right) +h\left(r^{adv}_{i,\vecbeta^*}\right) \right\}\vecX_i^\top \vecdelta_\eta\nonumber\\
&\quad \quad \quad -\lambda_o\sqrt{n}\sum_{i\in \mc{O}\cup(I_<\cap \mc{I})} \frac{1}{n}\left\{-h\left(r^\circ_{i,\vecbeta^*+\vecdelta_\eta}\right) +h\left(r^\circ_{i,\vecbeta^*}\right) \right\}\vecx_i^\top \vecdelta_\eta.
\end{align}
From \eqref{ine:det:1} - \eqref{ine:det:1:left}, we have
\begin{align}
\label{ine:det:main}
&\lambda_o\sqrt{n}\sum_{i=1}^n \frac{1}{n}\left\{-h\left(r^\circ_{i,\vecbeta^*+t(\hat{\vecbeta}-\vecbeta^*)}\right) +h\left(r^\circ_{i,\vecbeta^*}\right) \right\}\vecx_i^\top \vecdelta_\eta\nonumber\\
&\leq \lambda_o\sqrt{n}\sum_{i=1}^n\frac{1}{n} h\left(r^{\circ}_{i,\vecbeta^*}\right)\vecx_i^\top \vecdelta_\eta+\lambda_o\sqrt{n}\sum_{i \in \mc{O}}\hat{w}'_i h\left(r^{adv}_{i,\vecbeta^*+\vecdelta_\eta}\right) \vecX_i^\top\vecdelta_\eta-\lambda_o\sqrt{n}\sum_{i\in \mc{O}\cup(I_<\cap \mc{I})}\frac{1}{n} h\left(r^\circ_{i,\vecbeta^*+\vecdelta_\eta}\right)\vecx_i^\top \vecdelta_\eta
\end{align}
For each term of \eqref{ine:det:main},
from \eqref{ine:detassumption1}, \eqref{ine:detassumption2}, \eqref{ine:detassumption4} and \eqref{ine:detassumption5}, we have
\begin{align}
	&\left|\lambda_o\sqrt{n}\frac{1}{n}\sum_{i=1}^n h\left(r^\circ_{i,\vecbeta^*}\right)\vecx_i^\top \vecdelta_\eta\right|  \leq c_1 \|\vecdelta_\eta\|_2,\\
	&\left|\lambda_o\sqrt{n}\sum_{i\in \mc{O}\cup (I_<\cap\mc{I})}\frac{1}{n} h\left(r^\circ_{i,\vecbeta^*+\vecdelta_\eta}\right) \vecX_i^\top \vecdelta_\eta\right| \leq c_2 \|\vecdelta_\eta\|_2,\\
	&\left|\lambda_o\sqrt{n}\sum_{i \in \mc{O}} \hat{w}_i'h\left(r^\circ_{i,\vecbeta^*+\vecdelta_\eta}\right)\vecx_i^\top \vecdelta_\eta\right| \leq c_3 \|\vecdelta_\eta\|_2,\\
	&\lambda_o\sqrt{n}\sum_{i=1}^n \frac{1}{n}\left\{-h\left(r^\circ_{i,\vecbeta^*+\vecdelta_\eta}\right) +h\left(r^\circ_{i,\vecbeta^*}\right) \right\}\vecx_i^\top \vecdelta_\eta\geq c_4\|\vecdelta_\eta\|_2^2-c_5\|\vecdelta_\eta\|_2 -c_6.
\end{align}
Consequently, we have
\begin{align}
	c_4\|\vecdelta_\eta\|_2^2-c_5\|\vecdelta_\eta\|_2 -c_6 \leq (c_1+c_2+c_3)\|\vecdelta_\eta\|_2
\end{align}
and 
\begin{align}
\|\vecdelta_\eta\|_2 \leq \frac{c_1+c_2+c_3+c_5+\sqrt{(c_1+c_2+c_3+c_5)^2+4c_4c_6}}{2c_4} < 2\frac{c_1+c_2+c_3+c_5+\sqrt{c_4c_6}}{c_4} = r.
\end{align}
This contradicts $\|\vecdelta_\eta\|_2 = r$.
Consequently, we have $\hat{\vecbeta} \in \mbb{B}(r)$ and $\|\vecdelta\|_2 = \|\hat{\vecbeta}-\vecbeta^*\|_2 \leq r$.
\end{proof}\

\subsection{Proofs of Propositions \ref{p:main1}, \ref{p:main:out}, \ref{p:main:out2} and  \ref{p:main:sc}}
\label{sec:p}
\subsubsection{Some tools}
\label{tool}
In Section \ref{tool}, we introduce a useful tool used to prove Propositions \ref{p:main1}, \ref{p:main:out}, \ref{p:main:out2} and  \ref{p:main:sc}.
\begin{theorem}[Theorem 4 of \cite{KolLou2017Con}, Theorem 9.2.4 of \cite{Ver2018High}]
	\label{t:partialcov}
	Assume that $\{\vectau_i\}_{i=1}^n$  is a sequence with independent random vectors drawn from an $L$-subGaussian distribution with $\mbb{E}\tau_i =0$ and $\mbb{E}\tau_i \tau_i^\top = \Sigma$. Then, 
	we have
	\begin{align}
		\left\|\frac{1}{n}\sum_{i=1}^n \vectau_i \vectau_i^{\top} -\Sigma\right\|_{\mr{op}} \leq c(L) \sqrt{\|\Sigma\|_{\mr{op}}}  \left(\sqrt{\frac{r_\Sigma+t}{m}} + \frac{r_\Sigma+t}{m}\right)
	\end{align}
	with probability at least $1-e^{-t}$.
\end{theorem}

\subsubsection{Proof of proposition \ref{p:main1}}
	For any $\vecv_1,\vecv_2 \in \mbb{S}^{d-1}$, we have
	\begin{align}
		\left\|h\left(\frac{\xi_i}{\lambda_o\sqrt{n}}\right)\vecx_i^\top \vecv_1-h\left(\frac{\xi_i}{\lambda_o\sqrt{n}}\right)\vecx_i^\top \vecv_2\right\|_{\psi_2} &=\left\|h\left(\frac{\xi_i}{\lambda_o\sqrt{n}}\right)\vecz_i^\top \Sigma^\frac{1}{2}\vecv_1-h\left(\frac{\xi_i}{\lambda_o\sqrt{n}}\right)\vecz_i^\top\Sigma^\frac{1}{2} \vecv_2\right\|_{\psi_2} \nonumber \\
		&\leq L \|\Sigma^\frac{1}{2}\vecv_1-\Sigma^\frac{1}{2}\vecv_2\|_2
	\end{align}
	because $\left|h\left(\frac{\xi_i}{\lambda_o \sqrt{n}}\right)\right|\leq 1$.
	From the fact that $\left\{h\left(\frac{\xi_i}{\lambda_o\sqrt{n}}\right)\vecz_i^\top \Sigma^\frac{1}{2}\right\}_{i=1}^n$ is an i.i.d. sequence, we have
	\begin{align}
		\left\|\frac{1}{n}\sum_{i=1}^nh\left(\frac{\xi_i}{\lambda_o\sqrt{n}}\right)\vecz_i^\top \Sigma^\frac{1}{2}\vecv_1-\frac{1}{n}\sum_{i=1}^nh\left(\frac{\xi_i}{\lambda_o\sqrt{n}}\right)\vecz_i^\top \Sigma^\frac{1}{2}\vecv_2\right\|_{\psi_2} \leq L \|\Sigma^\frac{1}{2}\vecv_1-\Sigma^\frac{1}{2}\vecv_2\|_2.
	\end{align}
	From Exercise 8.6.5 of \cite{Ver2018High}, with probability at least $1-\delta$, 
	we have
	\begin{align}
		\label{gc-pre2}
		\sup_{\vecv \in \Sigma^\frac{1}{2}\mbb{S}^{d-1}}\left|\frac{1}{n} \sum_{i=1}^n h\left(\frac{\xi_i}{\lambda_o\sqrt{n}}\right) \vecz_i^\top \vecv  \right| \leq \frac{c(L)}{\sqrt{n}} \left(\mbb{E}\sup_{\vecv \in \Sigma^\frac{1}{2}\mbb{S}^{d-1}} \vecg_d^\top \vecv +\sqrt{\log(1/\delta)} \sup_{\vecu,\vecv \in \Sigma^\frac{1}{2}\mbb{S}^{d-1}}\sqrt{\|\vecu-\vecv\|^2_2}\right),
	\end{align}
	where $\vecg_d$ is the $d$-dimensional standard normal Gaussian vector.
 	From the proof of Theorem 9.2.4 of \cite{Ver2018High}, we have
	\begin{align}
		\mbb{E}\sup_{\vecv \in \Sigma^\frac{1}{2}\mbb{S}^{d-1}} \vecg_d^\top \vecv\leq \sqrt{\mr{Tr}(\Sigma)} 
	\end{align}
	and from simple algebra, we have
	\begin{align}
		\sup_{\vecu,\vecv \in \Sigma^\frac{1}{2}\mbb{S}^{d-1}}\sqrt{\|\vecu-\vecv\|^2_2} \leq 2 \sqrt{\|\Sigma\|_{\mr{op}}}.
	\end{align}
	Lastly, we have
	\begin{align}
		\sup_{\vecv \in \mbb{S}^{d-1}}\left|\frac{1}{n} \sum_{i=1}^n h\left(\frac{\xi_i}{\lambda_o\sqrt{n}}\right) \vecx_i^\top \vecv  \right| =\sup_{\vecv \in \Sigma^\frac{1}{2}\mbb{S}^{d-1}}\left|\frac{1}{n} \sum_{i=1}^n h\left(\frac{\xi_i}{\lambda_o\sqrt{n}}\right) \vecz_i^\top \vecv  \right|  \leq \frac{c(L)}{\sqrt{n}} \sqrt{\|\Sigma\|_{\mr{op}}} \left(\sqrt{r_\Sigma} +\sqrt{\log(1/\delta)} \right).
	\end{align}
	and let $\vecv$ be $\vecv / \|\vecv\|$, the proof is complete.

\subsubsection{Proof of Proposition \ref{p:main:out}}
	For any $\|\vecu\|_\infty \leq c$ and $\vecv \in \mbb{S}^{d-1}$, 
	we note that
	\begin{align}
		\label{ine:6-2}
		\left|\sum_{i \in \mc{O}}\hat{w}'_iu_i \tilde{\vecX}_i^\top\vecv  \right|^2 \stackrel{(a)}{\leq}  \sum_{i \in \mc{O}} \hat{w}'_iu_i^2 \sum_{i \in \mc{O}}\hat{w}'_i |\tilde{\vecX}_i^\top \vecv|^2 \stackrel{(b)}{\leq} c^2\frac{o}{n}\sum_{i \in \mc{O}}\hat{w}_i |\tilde{\vecX}_i^\top \vecv|^2,
	\end{align}
	where (a) follows from H{\"o}lder's inequality and  (b) follows from the fact that $\hat{w}_i'\leq 2\hat{w}_i$.
	From  \eqref{ine:6-2}, we have 
	\begin{align}
		\left|\sum_{i \in \mc{O}}\hat{w}'_iu_i \vecX_i^\top\vecv  \right| \leq c\sqrt{\frac{o}{n}}\sqrt{\sum_{i =1}^n\hat{w}_i (\vecX_i^\top\vecv)^2} \leq c\sqrt{\frac{o}{n}}\sqrt{\sup_{\vecv \in\mbb{S}^{d-1}}\sum_{i =1}^n\hat{w}_i (\vecX_i^\top\vecv)^2}
	\end{align}
	From the optimality of $\{\hat{w}_i\}_{i=1}^n$, we have
	\begin{align}
		\sup_{\vecv \in\mbb{S}^{d-1}}\sum_{i =1}^n\hat{w}_i (\vecX_i^\top \vecv)^2=\lambda_{\max} \left(\sum_{i =1}^n\hat{w}_i \vecX_i \vecX_i^\top \right)\stackrel{(a)}{\leq} \lambda_{\max} \left(\sum_{i \in \mc{I}}\frac{1}{n(1-\varepsilon)} \vecX_i \vecX_i^\top \right)&=\sup_{\vecv \in \mbb{S}^{d-1}} \sum_{i \in \mc{I}}\frac{1}{n-o} (\vecX_i^\top \vecv)^2\nonumber \\
		&=\sup_{\vecv \in \mbb{S}^{d-1}} \sum_{i \in \mc{I}}\frac{1}{n-o} (\vecx_i^\top \vecv)^2
	\end{align}
	where $(a)$ follows from the fact that $\{w^{\circ}_i\}_{i\in \mc{I}} \in \Delta^{n-1}$ with $w^{\circ}_i = \frac{1}{n(1-\varepsilon)}$.
	Moreover, from the positivity of $(\vecx_i^\top \vecv)$, we have
	\begin{align}
		\sup_{\vecv \in \mbb{S}^{d-1}} \sum_{i \in \mc{I}}\frac{1}{n-o} (\vecx_i^\top \vecv)^2\leq \sup_{\vecv \in \mbb{S}^{d-1}} \sum_{i=1}^n\frac{1}{n-o} (\vecx_i^\top \vecv)^2
	\end{align}
	From Theorem \ref{t:partialcov} and the triangular inequality, we have
	\begin{align}
		\left(\left\|\sum_{i=1}^n\frac{1}{n} \vecx_i \vecx_i^\top\right\|_{\mr{op}} =\right)\sup_{\vecv \in \mbb{S}^{d-1}}\sum_{i=1}^n\frac{1}{n} (\vecx_i^\top \vecv)^2&\leq c(L) \sqrt{\|\Sigma\|_{\mr{op}}}  \left(\sqrt{\frac{r_\Sigma+t}{n}} + \frac{r_\Sigma+t}{n}+1\right)\leq c(L) \sqrt{\|\Sigma\|_{\mr{op}}},
 	\end{align}
	 where we use the assumption $\frac{r_\Sigma+t}{n}\leq 1$
	Combining the arguments above, we have
	\begin{align}
		\left|\sum_{i \in \mc{O}}\hat{w}'_iu_i \vecX_i^\top\vecv  \right| \leq c(L)\sqrt{\frac{o}{n}}\sqrt{\|\Sigma\|_{\mr{op}}},
	\end{align}
	where we use $o> 2n$. Re-defining $\vecv = \vecv/\|\vecv\|_2(\in \mbb{S}^{d-1})$, the proof is complete.

\subsubsection{Proof of Proposition \ref{p:main:out2}}
	For any $\|\vecu\|_\infty \leq c$ and $\vecv \in \mbb{S}^{d-1}$, 
	from H{\"o}lder's inequality, we have 
	\begin{align}
		\sum_{i \in I_m}\frac{1}{n}u_i \vecx_i^\top\vecv    \leq \sqrt{\sum_{i \in I_m}\frac{1}{n}u_i^2 }\sqrt{\sum_{i \in I_m}\frac{1}{n}(\vecx_i^\top\vecv)^2 }\leq \sqrt{\sum_{i \in I_m}\frac{1}{n}u_i^2 }\sqrt{\sum_{i =1}^n\frac{1}{n} (\vecx_i^\top\vecv)^2 }\leq c\sqrt{\frac{m}{n}}\sqrt{\sum_{i =1}^n\frac{1}{n} (\vecx_i^\top\vecv)^2 }.
	\end{align}
	From Theorem \ref{t:partialcov} and the triangular inequality, we have
	\begin{align}
		\left(\left\|\sum_{i=1}^n\frac{1}{n} \vecx_i \vecx_i^\top\right\|_{\mr{op}} =\right)\sup_{\vecv \in \mbb{S}^{d-1}}\sum_{i=1}^n\frac{1}{n} (\vecx_i^\top \vecv)^2&\leq c(L) \sqrt{\|\Sigma\|_{\mr{op}}}  \left(\sqrt{\frac{r_\Sigma+t}{n}} + \frac{r_\Sigma+t}{n}+1\right)\leq c(L) \sqrt{\|\Sigma\|_{\mr{op}}},
 	\end{align}
	 where we use the assumption $\frac{r_\Sigma+t}{n}\leq 1$
	Combining the arguments above, we have
	\begin{align}
		\left|\sum_{i \in I_m}\hat{w}'_iu_i \vecx_i^\top\vecv  \right| \leq c(L)\sqrt{\frac{m}{n}}\sqrt{\|\Sigma\|_{\mr{op}}}.
	\end{align}
	Re-defining $\vecv = \vecv/\|\vecv\|_2(\in \mbb{S}^{d-1})$, the proof is complete.

\subsubsection{Proof of Proposition \ref{p:main:sc}}
	Let
	\begin{align}
		u_i = \frac{\xi_i}{\lambda_o \sqrt{n}}, \quad v_i = \frac{\vecx_i^\top \vecv}{\lambda_o \sqrt{n}}.
	\end{align}
 	The left-hand side of \eqref{ine:sc} divided by $\lambda_o^2$ can be expressed as
	\begin{align}
		\sum_{i=1}^n  \left\{-h\left(\frac{\xi_i -\vecx_i^\top \vecv}{\lambda_o\sqrt{n}}\right)+h \left(\frac{\xi_i}{\lambda_o\sqrt{n}}\right) \right\}\frac{\vecx_i^\top \vecv}{\lambda_o \sqrt{n}} = \sum_{i=1}^n  \left\{-h\left(u_i-v_i\right)+h \left(u_i\right) \right\}v_i
	\end{align}
	From the convexity of the Huber loss,  we have $ \left\{-h\left(u_i-v_i\right)+h \left(u_i\right) \right\}v_i>0$ and 
	\begin{align}
		&\sum_{i=1}^n  \left\{-h\left(u_i-v_i\right)+h \left(u_i\right) \right\}v_i \geq \sum_{i=1}^n  \left\{-h\left(u_i-v_i\right)+h \left(u_i\right) \right\}v_i\mr{I}_{E_i},
	\end{align}
	where $\mr{I}_{E_i}$ is the indicator function of the event
	\begin{align}
		E_i :=  \left( \left|u_i\right | \leq \frac{1}{2} \right)  \cap \left(   \left |v_i \right |  \leq  \frac{1}{2\lambda_o\sqrt{n}}   \right).
	\end{align}
	Define the functions
	\begin{align}
	\label{def:phipsi}
		\varphi(v) =\begin{cases}
		v^2   &  \mbox{ if }  |v | \leq  \frac{1}{2\lambda_o\sqrt{n}}\\
		(v-1/2)^2   &  \mbox{ if }  \frac{1}{2\lambda_o\sqrt{n}}\leq v  \leq  1/2 \\
		(v+1/2)^2   &  \mbox{ if }  -1/2\leq v  \leq  -\frac{1}{2\lambda_o\sqrt{n}}   \\
		0 & \mbox{ if } |v| >1/2
	\end{cases} ~\mbox{ and }~
		\psi(u) = I_{(|u| \leq 1/2 ) }.
	\end{align}
	Let   $f_i(\vecv) = \varphi(v_i) \psi(u_i)$ and we have
	\begin{align}
	\label{ine:huv-conv-f}
		\sum_{i=1}^n  \left\{-h\left(u_i-v_i\right)+h \left(u_i\right) \right\}v_i &\geq \sum_{i=1}^n  \left\{-h\left(u_i-v_i\right)+h \left(u_i\right) \right\}v_i\mr{I}_{E_i}\nonumber\\
		&= \sum_{i=1}^n  v_i^2\mr{I}_{E_i}\nonumber\\
		& \stackrel{(a)}{\geq} \sum_{i=1}^n  \varphi(v_i) \psi(u_i)=\sum_{i=1}^n f_i(\vecv),
	\end{align}
	where (a) follows from $\varphi(v) \leq v^2$ for $|v| \leq 1/2$. We note that 
	\begin{align}
	\label{ine:f-1/4}
		f_i(\vecv) \leq\varphi(v_i) \leq \min\left(\frac{(\vecx_i^\top \vecv)^2}{\lambda_o^2n},\frac{1}{4\lambda_o^2n}\right).
	\end{align}
	To bound $\sum_{i=1}^n f_i(\vecv)$ from below, for any fixed $ \vecv \in  \mc{R}$, we have
	\begin{align}
	\label{ine:fbelow}
		\sum_{i=1}^n f_i(\vecv)&\geq \mbb{E}f(\vecv) -\sup_{\vecv' \in \mc{R}}  \Big|\sum_{i=1}^n f_i(\vecv')-\mbb{E}\sum_{i=1}^n f_i(\vecv')\Big|.
	\end{align}
	Define the supremum of a random process indexed by $\mc{R}$:
	\begin{align}
	\label{ap:delta}
		\Delta  :=  \sup_{ \vecv' \in \mc{R}} \left| \sum_{i=1}^n f_i(\vecv') - \mbb{E}\sum_{i=1}^n f_i	(\vecv') \right| .  
	\end{align}
	From \eqref{ine:huv-conv-f} and \eqref{def:phipsi}, we have
	\begin{align}
	\label{ine:aplower:tmp}
		\mbb{E}\sum_{i=1}^n f_i(\vecv)&\geq \sum_{i=1}^n\mbb{E} \left|\frac{ \vecx_i^\top \vecv}{\lambda_o\sqrt{n}} \right| ^2 - \sum_{i=1}^n\mbb{E}\left|\frac{\vecx_i^\top \vecv}{\lambda_o\sqrt{n}} \right| ^2 I \bigg( \left|\frac{\vecx_i^\top \vecv}{\lambda_o\sqrt{n}} \right| \geq \frac{1}{2\lambda_o\sqrt{n}}   \bigg) -  \mbb{E}\left|\frac{ \vecx_i^\top \vecv}{\lambda_o\sqrt{n}} \right| ^2 I\left( \left|\frac{\xi_i}{\lambda_o\sqrt{n}} \right|> \frac{1}{2} \right).
	\end{align}
	We note that
	\begin{align}
	\label{ine:v2}
		\mbb{E}(\vecx_i^\top \vecv)^2 = \|\Sigma^\frac{1}{2}\vecv\|_2^2 \leq \|\Sigma\|_{\mr{op}} \|\vecv\|_2^2.
	\end{align}
	We evaluate the right-hand side of \eqref{ine:aplower:tmp} at each term.
	First, we have
	\begin{align}
	\label{ap:ine:cov1}
		\sum_{i=1}^n\mbb{E} \left|\frac{\vecx_i^\top \vecv}{\lambda_o\sqrt{n}} \right| ^2 I \bigg( \left|\frac{\vecx_i^\top \vecv}{\lambda_o\sqrt{n}} \right| \geq \frac{1}{2\lambda_o\sqrt{n}}   \bigg) 
		&\stackrel{(a)}{\leq} 	\sum_{i=1}^n\sqrt{\mbb{E} \left|\frac{\vecx_i^\top \vecv}{\lambda_o\sqrt{n}} \right| ^4 } \sqrt{\mbb{E}   \ I \bigg( \left|\frac{\vecx_i^\top \vecv}{\lambda_o\sqrt{n}} \right| \geq \frac{1}{2\lambda_o\sqrt{n}}   \bigg) }\nonumber\\
		&\stackrel{(b)}{=} 	\sum_{i=1}^n\sqrt{\mbb{E}\left|\frac{\vecx_i^\top \vecv}{\lambda_o\sqrt{n}} \right| ^4 } \sqrt{\mbb{P}  \left(  \left|\vecx_i^\top \vecv\right| \geq \frac{1}{2}\right) }\nonumber\\
		&\stackrel{(c)}{\leq}	\sum_{i=1}^n4\sqrt{\mbb{E} \left|\frac{\vecx_i^\top \vecv}{\lambda_o\sqrt{n}} \right| ^4 } \sqrt{ \mbb{E}(\vecx_i^\top \vecv)^4}\nonumber\\
		&=\frac{4}{\lambda_o^2} \mbb{E}(\vecx_i^\top \vecv)^4\nonumber\\
		&\stackrel{(d)}{\leq}\frac{4L^4}{\lambda_o^2} \{\mbb{E}(\vecx_i^\top \vecv)^2\}^2\nonumber\\
		&=\frac{4L^4}{\lambda_o^2}  \|\Sigma^\frac{1}{2}\vecv\|_2^4\nonumber\\
		&\leq\frac{4L^4}{\lambda_o^2}  \|\Sigma^\frac{1}{2}\vecv\|_2^2 \|\Sigma\|_{\mr{op}}\|\vecv\|_2^2\stackrel{(f)}{\leq} \frac{1}{3\lambda_o^2 }	\|\Sigma^\frac{1}{2}\vecv\|_2^2,
	\end{align}
		where (a) follows from H{\"o}lder's inequality, (b) follows from the relation between indicator function and expectation, (c) follows from Markov's inequality, (d) follows from Assumption \ref{a:1} and  (f) follows the definition of $\mc{R}$.
		Second,  we have
	\begin{align}
	\label{ap:ine:cov2}
		\sum_{i=1}^n\mbb{E} \left|\frac{\vecx_i^\top \vecv}{\lambda_o\sqrt{n}} \right| ^2 I\left( \left|\frac{\xi_i}{\lambda_o\sqrt{n}} \right|\geq \frac{1}{2} \right) 
		&\stackrel{(a)}{\leq} \sum_{i=1}^n\sqrt{\mbb{E} \left|\frac{\vecx_i^\top \vecv}{\lambda_o\sqrt{n}} \right| ^4}  \sqrt{\mbb{E}I\left( \left|\frac{\xi_i}{\lambda_o\sqrt{n}} \right|\geq \frac{1}{2} \right)}\nonumber\\
		&\stackrel{(b)}{\leq}\sum_{i=1}^n\sqrt{\mbb{E}\left|\frac{\vecx_i^\top \vecv}{\lambda_o\sqrt{n}} \right| ^4}  \sqrt{\mbb{P}  \left( \left|\frac{\xi_i}{\lambda_o\sqrt{n}} \right|\geq \frac{1}{2}  \right)}\nonumber\\	  
		&\stackrel{(c)}{\leq}\sum_{i=1}^n \sqrt{\frac{2}{\lambda_o\sqrt{n}}}\sqrt{\mbb{E} \left|\frac{  \vecx_i^\top \vecv}{\lambda_o\sqrt{n}} \right| ^4}  \sqrt{\mbb{E}|\xi_i|}\nonumber\\	  
		&\stackrel{(d)}{\leq}\sum_{i=1}^n \sqrt{ \frac{2\sigma}{\lambda_o\sqrt{n}} }\frac{2L^2}{\lambda_o^2n}  \mbb{E}(\vecx_i^\top \vecv)^2 \nonumber \\
		&\stackrel{(e)}{\leq} \frac{1}{3\lambda_o^2} 	\mbb{E}(\vecx_i^\top \vecv)^2\nonumber\\
		&=\frac{1}{3\lambda_o^2} 	\|\Sigma^\frac{1}{2}\vecv\|_2^2,
	\end{align}
	where (a) follows from H{\"o}lder's inequality, (b) follows from relation between indicator function and expectation, (c) follows from Markov's inequality, (d) follows from the assumption of Theorem  \ref{a:1} and  (e) follows from the definition of $\lambda_o$.
	Consequently, we have
	\begin{align}
	\label{ap:f_bellow}
		\mbb{E}\sum_{i=1}^nf_i(\vecv)&\geq \frac{1}{3\lambda_o^2}	\|\Sigma^\frac{1}{2}\vecv\|_2^2 \geq \frac{1}{3}\lambda^2_{\min}(\Sigma^\frac{1}{2}) \|\vecv\|_2^2
	\end{align}
	and
	\begin{align}
	\label{ap:h_bellow}
		\sum_{i=1}^n  \left\{-h\left(u_i-v_i\right)  -h \left(u_i\right) \right\}v_i  \geq \sum_{i=1}^nf_i(\vecv) \geq   \frac{\lambda^2_{\min}(\Sigma^\frac{1}{2})}{3\lambda_o^2}	\|\vecv\|_2^2-\Delta.
	\end{align}
	Next we evaluate the stochastic term $\Delta$ defined in \eqref{ap:delta}. 
	From \eqref{ine:f-1/4} and Theorem 3 of \cite{Mas2000Constants}, with probability at least $1-\delta$, we have
	\begin{align}
	\label{ine:delta}
		\Delta & \leq 2 \mbb{E} \Delta + \sigma_f \sqrt{8\log(1/\delta)} + 5\frac{\log(1/\delta)}{\lambda_o^2 n},
	\end{align}
	where $\sigma^2_f= \sup_{ \vecv \in \mc{R}} \sum_{i=1}^n\mbb{E}  \{f_i(\vecv)-\mbb{E}f_i(\vecv)\}^2$.
	About $\sigma_f$, we have
	\begin{align}
		\mbb{E}\{f_i(\vecv)-\mbb{E}f_i(\vecv)\}^2 \leq \mbb{E}f_i^2(\vecv) .
	\end{align}
	From  \eqref{ine:f-1/4}, we have
	\begin{align}
	\label{ap:ine:cov3}
		&\mbb{E}f_i^2(\vecv)\leq \mbb{E} \frac{(\vecx_i^\top \vecv)^4}{\lambda_o^4n^2}= \frac{L^4}{\lambda_o^4n^2} \|\vecv\|_2^4.
	\end{align}
	Combining this and \eqref{ine:delta}, we have
	\begin{align}
	\label{ap:delta_upper}
		\Delta &\leq 2 \mbb{E} \Delta+ \frac{1}{\lambda_o^2}L^2\sqrt{\|\Sigma\|_{\mr{op}}}\sqrt{8\frac{\log(1/\delta)}{n}}\|\vecv\|_2^2+ 5\frac{\log(1/\delta)}{\lambda_o^2 n}\nonumber\\
		&\leq 2 \mbb{E} \Delta+ \frac{1}{\lambda_o^2}\sqrt{\frac{\log(1/\delta)}{n}}\|\vecv\|_2+ 5\frac{\log(1/\delta)}{\lambda_o^2 n}.
	\end{align}
	From Symmetrization inequality (Lemma 11.4 of \cite{BouLugMas2013concentration}), we have  $\mbb{E}\Delta \leq 2   \,\mbb{E} \sup_{ \vecv \in \mc{R}} |  \mathbb{G}_{\vecv} |  $,
	where 
	\begin{align}
		\mbb{G}_{\vecv} := \sum_{i=1}^n 
		a_i \varphi \left( \frac{\vecx_i^\top \vecv}{\lambda_o\sqrt{n}} \right) \psi \left(\frac{\xi_i}{\lambda_o\sqrt{n}} \right),
	\end{align} 
	and $\{a_i\}_{i=1}^n$ is a sequence of i.i.d. Rademacher random variables which is independent of $\{\vecx_i,\xi_i\}_{i=1}^n$.
	We   denote $\mbb{E}^*$ as a conditional variance of $\left\{a_i\right\}_{i=1}^n$ given $\left\{\vecx_i,\xi_i\right\}_{i=1}^n$. From contraction principal (Theorem 11.5 of \cite{BouLugMas2013concentration}),  we  have
	\begin{align}
		&\mbb{E} ^*\sup_{\vecv\in\mc{R}} \left|     \sum_{i=1}^n a_i \varphi \left( \frac{\vecx_i^\top \vecv}{\lambda_o\sqrt{n}} \right) \psi \left(\frac{\xi_i}{\lambda_o\sqrt{n}} \right)    \right| \leq	\mbb{E}^* \sup_{\vecv\in\mc{R}} \left|  \sum_{i=1}^n   a_i \varphi \left( \frac{\vecx_i^\top \vecv}{\lambda_o\sqrt{n}} \right) \right|
	\end{align}	
	and from the basic property of the expectation, we have
	\begin{align}
		\mbb{E}\sup_{\vecv\in\mc{R}}  \left|     \sum_{i=1}^n a_i \varphi \left( \frac{\vecx_i^\top \vecv}{\lambda_o\sqrt{n}} \right) \psi \left(\frac{\xi_i}{\lambda_o\sqrt{n}} \right)    \right| &\leq	
		\mbb{E}\sup_{\vecv\in\mc{R}}  \left|  \sum_{i=1}^n   a_i \varphi\left( \frac{\vecx_i^\top \vecv}{\lambda_o\sqrt{n}} \right) \right|.
	\end{align}	
	Since $\varphi$ is $\frac{1}{2\lambda_o\sqrt{n}}$-Lipschitz and $\varphi(0)=0$,  from contraction principal (Theorem 11.6 in \cite{BouLugMas2013concentration}), we have
	\begin{align}
		\mbb{E} \sup_{\vecv\in\mc{R}} \left|\sum_{i=1}^n a_i \varphi \left( \frac{\vecx_i^\top \vecv}{\lambda_o\sqrt{n}} \right)\right|&\leq	\mbb{E}\sup_{\vecv\in\mc{R}}  \left|  \sum_{i=1}^n   a_i  \frac{\vecx_i^\top \vecv}{2\lambda_o^2n}  \right|.
	\end{align}
	From Lemma \ref{p:1e} and the definition of $\mc{R}$, we have
	\begin{align}
	\label{ine:hub-stoc-upper}
		2\lambda_o^2\mbb{E} \Delta \leq c(L)\sqrt{\|\Sigma\|_{\mr{op}}}\sqrt{\frac{r_\Sigma}{n}}\|\vecv\|_2.
	\end{align}
	Combining \eqref{ine:hub-stoc-upper}  with \eqref{ap:delta_upper}  and \eqref{ap:h_bellow}, with probability at least $1-\delta$, we have
	\begin{align}
		&\sum_{i=1}^n \lambda_o^2 \left\{-h\left(\frac{\xi_i + \vecx_i^\top \vecv}{\lambda_o\sqrt{n}}\right)+h \left(\frac{\xi_i}{\lambda_o\sqrt{n}}\right) \right\}\frac{\vecx_i^\top \vecv}{\lambda_o \sqrt{n}}\nonumber\\
		&\geq \frac{\lambda^2_{\min}(\Sigma^\frac{1}{2})}{3}\|\vecv\|_2^2-C\left(L\sqrt{\|\Sigma\|_{\mr{op}}} \sqrt{\frac{r_\Sigma}{n}}+\sqrt{\frac{\log(1/\delta)}{n}} \right)\|\vecv\|_2 -C\frac{\log(1/\delta)}{n}\nonumber\\
		&\stackrel{(a)}{\geq} \frac{\lambda^2_{\min}(\Sigma^\frac{1}{2})}{3}\|\vecv\|_2^2-CL\left(\sqrt{\|\Sigma\|_{\mr{op}}} \sqrt{\frac{r_\Sigma}{n}}+\sqrt{\frac{\log(1/\delta)}{n}} \right)\|\vecv\|_2 -CL\frac{\log(1/\delta)}{n},
	\end{align} 
	where (a) follows from $L\geq 1$.

\begin{lemma}
	\label{p:1e}
	Suppose that Assumption \ref{a:1} holds.
	Then, we have
	\begin{align}
		\label{ine:gc-normale}
		&\mbb{E}\left\|\frac{1}{n}\sum_{i=1}^n a_i \vecx_i \right\|_2 \leq c(L)\sqrt{\|\Sigma\|_{\mr{op}}} \sqrt{\frac{r_\Sigma}{n}} .
	\end{align}
\end{lemma}
	From the definition of $\{a_i\}_{i=1}^n$ and assumption on $\{\vecx_i\}_{i=1}^n$, we have
	\begin{align}
		\left\|\frac{1}{\sqrt{n}}\sum_{i=1}^n a_i \vecx_i^\top \vecv_1-\frac{1}{\sqrt{n}}\sum_{i=1}^na_i\vecx_i^\top \vecv_2\right\|_{\psi_2} = \left\|\frac{1}{\sqrt{n}}\sum_{i=1}^n a_i \vecz_i^\top \Sigma^\frac{1}{2}\vecv_1-\frac{1}{\sqrt{n}}\sum_{i=1}^na_i\vecz_i^\top \Sigma^\frac{1}{2} \vecv_2\right\|_{\psi_2}\leq L\|\Sigma^\frac{1}{2}(\vecv_1-\vecv_2)\|_2.
	\end{align}
	for any $\vecv_1,\vecv_2 \in \mbb{S}^{d-1}$.
	From corollary 8.6.3 of \cite{Ver2018High}, we have
	\begin{align}
		\mbb{E}\left\|\frac{1}{n}\sum_{i=1}^n a_i \vecx_i \right\|_2  =\mbb{E}\sup_{\vecv \in \mbb{S}^{d-1}}\frac{1}{n}\sum_{i=1}^n a_i \vecx_i^\top \vecv \leq CL\frac{1}{\sqrt{n}} \mbb{E} \sup_{\vecv \in \Sigma^\frac{1}{2}\mbb{S}^{d-1}}\vecg_d^\top \vecv,
	\end{align}
	where $\vecg_d$ is the $d$-dimensional standard normal Gaussian random vector.
 	From the proof of Theorem 9.2.4 of \cite{Ver2018High}, we have
	\begin{align}
		\mbb{E}\sup_{\vecv \in \Sigma^\frac{1}{2}\mbb{S}^{d-1}} \vecg_d^\top \vecv\leq \sqrt{\mr{Tr}(\Sigma)} 
	\end{align}
	and the proof is complete.

\bibliographystyle{plain}
\bibliography{ARWHR} 
\end{document}